\documentclass[12pt]{amsart}
\usepackage{amsmath,amsfonts,amsthm,amssymb,mathrsfs,amsxtra,amscd,latexsym}
\usepackage[hmargin=2cm,vmargin=3cm]{geometry}
\usepackage[dvipdfm, usenames]{color}
\usepackage{xcolor}
 \newtheorem{thm}{Theorem}[section]

 \newtheorem{lem}[thm]{Lemma}

 \newtheorem{dfn}[thm]{Definition}

 \newtheorem{rmk}[thm]{Remark}
 \newtheorem{ex}[thm]{Example}
 \theoremstyle{definition}
 \theoremstyle{remark}
 \numberwithin{equation}{section}

\newcommand{\sm}{\left(\begin{smallmatrix}}
\newcommand{\esm}{\end{smallmatrix}\right)}
\newcommand{\mat}{\left(\begin{matrix}}
\newcommand{\emat}{\end{matrix}\right)}
\newcommand{\lla}{\,\langle\!\langle\,}
\newcommand{\rra}{\,\rangle\!\rangle\,}
\newcommand{\la}{\,\langle\,}
\newcommand{\ra}{\,\rangle\,}

\newcommand{\mbf}{\mathbf}

\newcommand{\mc}{\mathcal}

\def\CC{\mathbb{C}}
\def\HH{\mathbb{H}}

\def\ZZ{\mathbb{Z}}

\def\cusp{\mathrm{cusp}}

\def\m{\mathrm{mod}}

\def\sk{\mathop{\rm sk}}

\def\GL{\mathrm{GL}}
\def\Ind{\mathrm{Ind}}

\def\SL{\mathrm{SL}}

\title[Pairings of harmonic Maass-Jacobi forms]{Pairings of harmonic Maass-Jacobi forms involving special values of partial $L$-functions} 

 \author{Dohoon Choi}
 \author{Subong Lim}

 \address{School of liberal arts and sciences, Korea Aerospace University, 200-1, Hwajeon-dong, Goyang, Gyeonggi 412-791, Republic of Korea}
  \email{choija@kau.ac.kr}

 \address{School of Mathematics, Korea Institute for Advanced Study (KIAS), 85 Hoegiro, Dongdaemun-gu, Seoul 130-722, Republic of Korea}
  \email{subong@kias.re.kr}



\begin{document}
\begin{abstract}
In this paper, considering  the Eichler-Shimura cohomology theory for Jacobi forms, we study connections between harmonic Maass-Jacobi forms and Jacobi integrals.
As an application we study a pairing between two Jacobi integrals, which is defined by special values of partial $L$-functions of skew-holomorphic Jacobi cusp forms. We obtain connections between this pairing and the Petersson inner product for skew-holomorphic Jacobi cusp forms.
This result can be considered as analogue of Haberland formula of elliptic modular forms for Jacobi forms.
\end{abstract}

\maketitle

\section{Introduction}

Zwegers \cite{Zwe} obtained one of  important advances in researches on mock theta functions.
This result implies that  mock theta functions can be interpreted as specializations of certain real-analytic Jacobi forms to some points.
The result of Zwegeres motivated systematic studies in
a harmonic Maass Jacobi form \cite{BR} and a Jacobi integral \cite{CL2}.
Bringmann and Richter \cite{BR} suggested and studied a systematic theory for harmonic Maass-Jacobi forms, including real analytic Jacobi forms constructed by Zwegers, and from this they proved Zagier type dualities for Jacobi forms. A Jacobi integral, which is analogue of an Eichler integral, was studied by Choie and the second author \cite{CL2}, and they showed that Zwegers' real analytic Jacobi forms can be expressed as a sum of two Jacobi integrals.


In this paper, we study connections between a harmonic Maass-Jacobi form  and a Jacobi integral.  More precisely we show that the holomorphic part of a harmonic Maass-Jacobi form is a Jacobi integral and that the non-holomorphic part of a harmonic Maass-Jacobi form is given by the non-holomorphic Jacobi integral of a skew-holomorphic Jacobi cusp form.
In order to do it we study the supplementary function theory for Jacobi forms. 
The supplementary function theory for elliptic modular forms  was introduced by Knopp \cite{Kno} and developed by Knopp and Husseini \cite{HK} to study Eichler integrals and Eichler-Shimura cohomology for Fuchsian groups.
Furthermore, we  prove analogues of Haberland formula in the context of Jacobi forms.

Throughout this paper $\Gamma_1$ denotes $\SL_2(\ZZ)$.
Let $\Gamma\subset\Gamma_1$ be a $H$-group, i.e., a finitely generated Fuchsian group of the first kind, which has at least one parabolic class. We assume that $[\Gamma_1:\Gamma]<\infty$.
Let $k\in\ZZ$, $m\in\ZZ$ with $m>0$ and $\chi$ be a  multiplier system of weight $k+\frac 12$ on $\Gamma$.
 We denote the space of Jacobi cusp forms (resp. skew-holomorphic Jacobi cusp forms) of weight $k+\frac12$, index $m$ and multiplier system $\chi$ on $\Gamma^{J}$ by $S_{k+\frac12,m,\chi}(\Gamma^{J})$ (resp. $J^{\sk,\cusp}_{k+\frac12,m,\chi}(\Gamma^{J})$), where   $\Gamma^{J}$ denotes a  Jacobi group $\Gamma \ltimes \ZZ^2$.

First we review some basic notions of harmonic Maass-Jacobi forms based on \cite{BR}. A harmonic Maass-Jacobi form of weight $k+\frac12$, index $m$ and multiplier system $\chi$ on $\Gamma^J$ is a smooth function $F$ on $\HH\times\CC$  that transforms like a Jacobi form of weight $k+\frac12$, index $m$ and multiplier system  $\chi$ on $\Gamma^J$ and is annihilated by the differential operator $C^{k+\frac12, m}$, where $C^{k,m}$ is given by
\begin{eqnarray*}
C^{k,m} &:=& -2(\tau-\bar{\tau})^2\partial_{\tau\bar{\tau}} -(2k-1)(\tau-\bar{\tau})\partial_{\bar{\tau}} + \frac{(\tau-\bar{\tau})^2}{4\pi im}\partial_{\bar{\tau}zz}\\
&& + \frac{k(\tau-\bar{\tau})}{4\pi im}\partial_{z\bar{z}} + \frac{(\tau-\bar{\tau})(z-\bar{z})}{4\pi im}\partial_{zz\bar{z}}-2(\tau-\bar{\tau})(z-\bar{z})\partial_{\tau\bar{z}} + k(z-\bar{z})\partial_{\bar{z}}\\
&&+\frac{(\tau-\bar{\tau})^2}{4\pi im}\partial_{\tau\bar{z}\bar{z}}+ \biggl(\frac{(z-\bar{z})^2}{2} + \frac{k(\tau-\bar{\tau})}{4\pi im}\biggr)\partial_{\bar{z}\bar{z}} + \frac{(\tau-\bar{\tau})(z-\bar{z})}{4\pi im}\partial_{z\bar{z}\bar{z}}.
\end{eqnarray*}
Here we write  $\partial_\tau := \frac{\partial}{\partial\tau}, \partial_{\bar{\tau}} := \frac{\partial}{\partial\bar{\tau}},\partial_z := \frac{\partial}{\partial z}$ and $\partial_{\bar{z}} := \frac{\partial}{\partial\bar{z}}$, where $\tau = u+iv\in\HH$ and $z = x+iy\in\CC$.
We are particularly interested in harmonic Maass-Jacobi forms, which are holomorphic in $z$. We denote the space of such forms by $\hat{J}_{k+\frac12,m,\chi}(\Gamma^{J})$.

We introduce another differential operator
\[
\xi_{k,m} := \biggl(\frac{\tau-\bar{\tau}}{2i}\biggr)^{k-5/2}D^{(m)}_-,
\]
where
\[
D_-^{(m)} := \biggl(\frac{\tau-\bar{\tau}}{2i}\biggr)\biggl(-(\tau-\bar{\tau})\partial_{\bar{\tau}}-(z-\bar{z})\partial_{\bar{z}}+\frac{1}{4\pi m}\biggl(\frac{\tau-\bar{\tau}}{2i}\biggr)\partial_{\bar{z}\bar{z}}\biggr).
\]
The assignment $F\mapsto \xi_{k+\frac12,m}(F)$ gives a linear map
\[
\xi_{k+\frac12,m}: \hat{J}_{k+\frac12,m,\chi}(\Gamma^{J}) \to J^{\sk!}_{\frac52-k,m,\chi}(\Gamma^{J}),
\]
where $J^{\sk!}_{\frac52-k,m,\chi}(\Gamma^{J})$ is the space of weak skew-holomorphic Jacobi forms of weight $\frac52-k$, index $m$ and multiplier system $\chi$ on $\Gamma^J$.
We denote the pre-image of the space of skew-holomorphic Jacobi cusp forms $J^{\sk,\cusp}_{\frac52-k,m,\chi}(\Gamma^{J})$ under $\xi_{k+\frac12,m}$ by $\hat{J}^{\cusp}_{k+\frac12,m,\chi}(\Gamma^{J})$.
Any harmonic  Maass-Jacobi form $F\in \hat{J}^{cusp}_{k+\frac12,m,\chi}(\Gamma^J)$ has a unique decomposition $F =F^+ + F^-$, where the function $F^{+}$ (resp. $F^{-}$) is called the holomorphic (resp. non-holomorphic) part of $F$. We denote   the space of holomorphic parts of $F\in \hat{J}^{cusp}_{k+\frac12,m,\chi}(\Gamma^J)$ by $\hat{J}^+_{k+\frac12,m,\chi}(\Gamma^J)$.

In the theory of harmonic weak Maass forms there are two important differential operators $\xi_{k} := 2i(\frac{\tau-\bar{\tau}}{2i})^{k}\overline{\partial_{\overline{\tau}}}$  and $D:= \frac{1}{2\pi i}\partial_{\tau}$ (see \cite{BF, BOR}). The differential operator $\xi_{k,m}$ is analogous with $\xi_k$ and we have another differential operator, which is analogous with $D$, given by
\[
H_m :=  8\pi im\partial_{\tau} - \partial_{zz}.
\]
This is called the heat operator.
The action of the heat operator on Jacobi forms was  studied by other researchers (for example see \cite{BRR, CK}).


Now we consider the  special $\mathbb{C}[\Gamma^{J}]$-module to define a Jacobi integral.
\begin{dfn}\label{dfnofpme}
For  positive integers $k$ and $m$  let $P_{k,m}$ be the set of holomorphic functions $G$ on $\HH\times\CC$ which satisfy the  conditions
\begin{enumerate}
\item[(1)] for all $X\in\ZZ^2$, $G|_{m}X = G$ (for the definition of $|_{m}X$ see Section \ref{section3.1}),
\item[(2)] $H^{k+1}_{m}G = 0$.
\end{enumerate}
\end{dfn}{

The set $P_{k-2,m}$ is preserved under the slash operator $|_{\frac 52-k,m,\chi}(\gamma,X)$ for a multiplier system $\chi$ of weight $\frac52-k$ and  $(\gamma,X)\in \Gamma_1^{J}$, and it forms a vector space over $\CC$.
A function $F$ on $\HH\times\CC$ is called a Jacobi integral of weight $\frac 52-k$, index $m$ and multiplier system $\chi$ on $\Gamma^{J}$ if it satisfies
\[
F-F|_{\frac52-k,m,\chi}(\gamma,X) \in P_{k,m}
\]
for all $(\gamma,X)\in\Gamma^J$.
 We let $E_{\frac52-k,m,\chi}(\Gamma^J)$ denote the space of holomorphic Jacobi integrals $F$ of weight $\frac52-k$, index $m$ and multiplier system $\chi$ on $\Gamma^J$ such that
\begin{enumerate}
\item functions $F$ are invariant under $|_{\frac52-k, m,\chi}(Q,0)$ for a generator $Q$ of $\Gamma_\infty$,
\item functions $H_m^{k-1}(F)$ can be written as
\[
H_m^{k-1}(F)  = G^* + H_m^{k-1}(H),
\]
where $G^*$ is a supplementary function to a skew-holomorphic Jacobi cusp form  $G \in J^{\sk,\cusp}_{k+\frac12,m,\chi}(\Gamma^J)$ and $H \in J^!_{\frac52-k,m,\chi}(\Gamma^J)$ (for the definition of a supplementary function, see Section \ref{section3.2}).
\end{enumerate}

We define the set
\[
\mc{N}_{\chi} := \{a\in\mc{N}|\ \chi(Q)e^{-2\pi i\lambda ma^2} = 0\},
\]
where $\mc{N} = (2m)^{-1}\ZZ/\ZZ$ and $Q = \sm 1&\lambda\\ 0&1\esm, \lambda>0$, is a generator of $\Gamma_\infty$.
The following theorem shows that we can understand holomorphic parts and non-holomorphic parts of harmonic Maass-Jacobi forms in terms of Jacobi integrals.

\begin{thm} \label{main2}  Let $\hat{J}^+_{\frac52-k, m,\chi}(\Gamma^J)$ and $E_{\frac52-k,m,\chi}(\Gamma^J)$ be as above.
\begin{enumerate}
\item For a positive integer $k$ with $k>2$
\begin{equation*}
E_{\frac52-k,m,\chi}(\Gamma^J)=\hat{J}^+_{\frac52-k, m,\chi}(\Gamma^J) + C_{m,\chi},
\end{equation*}
where $C_{m,\chi}$ is the space of  functions on $\HH\times\CC$ generated by $\theta_{m,a}$ for $a\in \mc{N}_{\chi}$ (for the definition of $\theta_{m,a}$ see Section \ref{section3.1}).

\item If $F\in \hat{J}^{\cusp}_{\frac52-k,m,\chi}(\Gamma^J)$, then
\begin{equation} \label{shadow}
F^- = \hat{\mc{E}}^J(G),
\end{equation}
where $G$ is a skew-holomorphic Jacobi cusp form in $J^{\sk,\cusp}_{k+\frac12, m,\chi}(\Gamma^J)$ and $\hat{\mc{E}}^J(G)$ is a non-holomorphic Jacobi integral associated with $G$ (for the precise definition of $\hat{\mc{E}}^J(G)$ see Section \ref{section3.2}).
\end{enumerate}
\end{thm}

For $F\in \hat{J}^{\cusp}_{\frac52-k,m,\chi}(\Gamma^J)$, $F^+$ is called a mock Jacobi form.
If a skew-holomorphic Jacobi cusp form $G$ satisfies (\ref{shadow}), then $G$ is called the shadow of a mock Jacobi form $F^+$ (see  \cite{BR} and \cite{ABS}).

\begin{rmk}
\begin{enumerate}
\item The main part of the proof of Theorem \ref{main2} is the construction of a harmonic Maass-Jacobi form from a given Jacobi integral $F\in E_{\frac52-k,m,\chi}(\Gamma^J)$.
\item An Eichler integral, which was first studied by Eichler \cite{Eic} in  connection with the Eichler-Shimura cohomology theory, is an important tool to understand a harmonic weak Maass form.
More precisely, any harmonic weak Maass form can be written as a sum of holomorphic part and non-holomorphic part and
both of these two parts are Eichler integrals of modular forms with the dual weight (for example see \cite{CKL}).
Theorem \ref{main2} is analogue of this in the case of Jacobi forms.
\end{enumerate}
\end{rmk}

Haberland \cite{Hab} established a formula to express the Petersson inner product of two cusp forms on $\Gamma_1$ as the sum of their special $L$-values.
Kohnen and Zagier \cite{KZ} interpreted this formula by using period polynomials, and Pa\c{s}ol and Popa \cite{PP} generalized this approach to arbitrary modular forms for congruence subgroups.
Recently, Bringmann, Guerzhoy, Kent and Ono \cite{BGK} gave interesting relations between period functions of harmonic weak Maass forms, which follow from the connection between the obstruction to modularity and the period functions of the associated Eichler integral.
With the notion of partial $L$-functions of a skew-holomorphic Jacobi cusp form, we obtain the similar result in the context of Jacobi forms.

We construct a pairing $\{\ ,\ \}$ on mock Jacobi forms and prove that the pairing $\{\ ,\ \}$ can be written as the Petersson inner product of their shadows.
Suppose that $\Gamma$ contains $-I$.
For skew-holomorphic Jacobi cusp forms $F$ and $G$ of weight $k$ the Petersson inner product is defined by
\[
(F,G) := \frac{1}{[\Gamma_1:\Gamma]} \int_{\Gamma^J\setminus \HH\times\CC} F(\tau,z)\overline{G(\tau,z)} v^{k-3}e^{-4\pi my^2/v}dxdydudv.
\]

Now we consider a vector-valued Jacobi integral and explain how to define a pairing on mock Jacobi forms.
Note that $P_{k-2,m}$ is a $\Gamma^J$-module.
Let $\widetilde{P_{k-2,m}}$ be the induced $\Gamma_1^J$-module $\Ind^{\Gamma_1^J}_{\Gamma^J}(P_{k-2,m})$.
If the slash operator $|_{\frac52-k,m,\chi}$ can be extended to $\Gamma_1^J$, then  $P_{k-2,m}$ is also a $\Gamma_1^J$-module and we can identify $\widetilde{P_{k-2,m}}$ with the space of maps $P: \Gamma\setminus \Gamma_1 \to P_{k-2,m}$ with $\Gamma_1^J$ action
\[
(P\|_{\frac52-k, m,\chi} (\gamma,X))(A) := P(A\gamma^{-1})|_{\frac52-k,m,\chi}(\gamma,X)
\]
for $(\gamma,X)\in \Gamma_1^J$ and $A\in\Gamma_1$.


By Lemma 3.1 in \cite{CL1}, every function $G\in P_{k-2,m}$ can be written as
\[
G(\tau,z) = \sum_{a\in\mathcal{N}} g_a(\tau)\theta_{m,a}(\tau,z),
\]
where $g_a$ is a polynomial of degree at most $k-2$ for each $a\in\mathcal{N}$.
Let
\[
G^c(\tau,z) = \sum_{a\in\mathcal{N}} \overline{g_a(\overline{\tau})}\theta_{m,a}(\tau,z).
\]
On $P_{k-2,m}\times P_{k-2,m}$ we have a pairing
\begin{equation} \label{pairing1}
\bigg\langle \sum_{a\in\mathcal{N}} f_a(\tau)\theta_{m,a}(\tau,z), \sum_{a\in\mathcal{N}} g_a(\tau)\theta_{m,a}(\tau,z) \bigg\rangle := \sum_{a\in\mathcal{N}} \sum_{n=0}^{k-2} (-1)^{k-2-n}\mat k-2\\n\emat^{-1} A_{a,n} B_{a,k-2-n},
\end{equation}
where $f_a(\tau) = \sum_{n=0}^{k-2} A_{a,n}\tau^n$ and $g_a(\tau) = \sum_{n=0}^{k-2}B_{a,n}\tau^n$.
With the pairing in (\ref{pairing1}) we define a pairing on $\widetilde{P_{k,m}} \times \widetilde{P_{k,m}}$ by
\[
\lla P,Q\rra := \frac{1}{[\Gamma_1:\Gamma]} \sum_{A\in\Gamma\setminus\Gamma_1} \la P(A), Q(A) \ra
\]
for $P,Q \in \widetilde{P_{k,m}}$.

Suppose that $F\in \hat{J}^{\cusp}_{\frac52-k,m,\chi}(\Gamma^J)$.
For each $\gamma\in\Gamma_1$ we have a unique decomposition
\[F|_{\frac52-k, m,\chi}(\gamma,0) = F_\gamma^+ + F_\gamma^-,\]
where $F_\gamma^+$ is the holomorphic part and $F_\gamma^-$ is the non-holomorphic part.
Let $A_{\frac52-k, m,\chi}(\Gamma^J)$ be the space of Jacobi integrals of weight $\frac52-k$, index $m$ and multiplier system $\chi$ on $\Gamma^J$  which have a theta expansion.
We define a vector-valued Jacobi integral $\widetilde{F^+}:\Gamma\setminus\Gamma_1\to A_{\frac52-k,m,\chi}(\Gamma^J)$ associated with $F^+$ by
\begin{equation} \label{plus1}
\widetilde{F^+}(\gamma) = F_\gamma^+.
\end{equation}
Its period function is given by
\begin{equation} \label{plus2}
P(F^+) = \widetilde{F^+} - \widetilde{F^+}\|_{\frac52-k,m,\chi}(S,0),
\end{equation}
where $S = \sm 0&-1\\1&0\esm$.
With this we define a pairing on $\hat{J}^{+}_{\frac52-k,m,\chi}(\Gamma^J) \times \hat{J}^{+}_{\frac52-k,m,\chi}(\Gamma^J)$  by
\begin{equation} \label{pairing}
\{F^+,G^+\} = \lla P(F)\|_{\frac52-k,m,\chi}((T^{-1},0)-(T,0)), P(G)^c\rra,
\end{equation}
where 
$T = \sm 1&1\\ 0&1\esm$.
This is analogous with the pairing defined by Pa\c{s}ol and Popa \cite{PP} for elliptic modular forms.
The following theorem shows that this pairing can be written in terms of the Petersson inner product of skew-holomorphic Jacobi cusp forms, and it can also be expressed as the sum of partial $L$-values of skew-holomorphic Jacobi cusp forms.

\begin{thm} \label{Haberland}
Let $k$ be a positive integer with $k\geq2$ and $\chi$ be a multiplier system of weight $k+\frac12$ on $\Gamma_1$.
Suppose that $\Gamma$ contains $-I$.
For $F,G\in \hat{J}^{\cusp}_{\frac52-k,m,\chi}(\Gamma^J)$ let $F'$ and $G'$ be the shadows of $F^+$ and $G^+$, respectively. Then
\begin{align} \label{peter}
&-6\sqrt{4m}(2i)^{k-1}[\Gamma_1:\Gamma](F',G') = \{F^+,G^+\}\\
\nonumber & \qquad =  \sum_{A\in\Gamma\setminus\Gamma_1}\sum_{a\in\mathcal{N}} \sum_{0\leq m+n\leq k-2\atop m\neq\equiv n\ (\m\ 2)} \frac{(k-2)!}{(k-2-m-n)!}\frac{i^{m+n}}{(2\pi)^{m+n+2}}\\
\nonumber&\qquad \times\biggl((-1)^{k-2-m} L^{\sk}(G'|^{\sk}_{k+\frac12,m,\chi}A,a,n+1)\overline{L^{\sk}(F'|^{\sk}_{k+\frac12,m,\chi}AT,a,m+1)}\\
\nonumber&\qquad \qquad-(-1)^m L^{\sk}(G'|^{\sk}_{k+\frac12,m,\chi}A,a,n+1)\overline{L^{\sk}(F'|^{\sk}_{k+\frac12,m,\chi}AT^{-1},a,m+1)}\biggr) .
\end{align}
\end{thm}

Here $L^{\sk}(F',a,s)$ is a partial $L$-function of a skew-holomorphic Jacobi cusp form $F'\in J^{\sk,\cusp}_{k+\frac51, m,\chi}(\Gamma^J)$ defined as follows.
Suppose that $F'$ has the Fourier expansion of the form
\[
F'(\tau,z) = \sum_{4m(n+\kappa)-\lambda r^2<0}c\biggl(\frac{n+\kappa}\lambda,r\biggr)e^{\frac{-\pi(\lambda r^2-4m(l+\kappa))v}{m\lambda}}q^{(n+\kappa)/\lambda}\zeta^{r}.
\]
By the theta expansion we have
\[
F'(\tau,z)  = \sum_{\mu (\m\ 2m)} \sum_{r\equiv \mu (\m 2m)} \sum_{n\in\ZZ\atop M(n)>0} \overline{D_\mu(M(n))q^{\frac{M(n)}{4m}}}q^{\frac{r^2}{4m}}\zeta^r,
\]
where $M(n) = r^2 - \frac{4m(l+\kappa)}{\lambda}$ and $D_\mu(M(n)) = \overline{c(\frac{r^2-M(n)}{4m},r)}$.
Then partial $L$-functions of $F'$ are defined by
\[
L^{\sk}(F',a,s) = \sum_{n\in\ZZ\atop M(n)>0} \frac{D_\mu(M(n)) } {(M(n)/4m)^s}.
\]

\begin{ex}
Suppose that $m=1$ and $\Gamma = \Gamma_1$.
Let
\[
W_{\frac52-k,m,\chi} := \bigg\{P\in P_{k-2,m}\bigg|\ P|_{\frac52-k,m,\chi}(1+S) = P|_{\frac52-k,m,\chi}(1+U+U^2) = 0\bigg\},
\]
where $U = TS$.
By the Eichler-Shimura cohomology theory for Jacobi forms we have an isomorphism (see \cite{CL1})
\begin{equation} \label{iso1}
H^1_{\frac52-k,m,\chi}(\Gamma^J, P_{k-2,m}) \cong S_{k+\frac12,m,\chi}(\Gamma^J) \oplus J^{\sk,\cusp}_{k+\frac12,m,\chi}(\Gamma^J).
\end{equation}
As in the case of elliptic modular forms, the isomorphism (\ref{iso1}) induces the  isomorphism
\begin{equation} \label{iso2}
W_{\frac52-k,m,\chi} \cong S_{k+\frac12,m,\chi}(\Gamma^J) \oplus J^{\sk,\cusp}_{k+\frac12,m,\chi}(\Gamma^J).
\end{equation}

Note that in the case $\Gamma = \Gamma_1$ there are exactly $6$ multiplier systems for each weight and they are given in terms of the Dedekind eta function $\eta$, which is given by
\[
\eta(\tau) = q^{\frac1{24}}\prod_{n=1}^\infty (1-q^n),
\]
where $q= e^{2\pi i\tau}$.
We define a multiplier system $\chi_i$ by
\[
\chi_i(\gamma) = \frac{\eta^i(\gamma\tau)}{(c\tau+d)^{\frac i2}\eta^i(\tau)}
\]
for $\gamma = \sm a&b\\c&d\esm\in\Gamma_1$.
Note that $\chi_i = \chi_j$ if $i\equiv j\ (\m\ 24)$ since $\chi_{24}$ is trivial.
If the weight  $k$ is in $\frac12\ZZ$, then the corresponding multiplier systems are
\[
\{\chi_i|\ i\equiv 2k\ (\m\ 4)\}.
\]

We  compute $W_{\frac52-k,m,\chi_i}$ in the case when $k = 2,3$ and $4$.
If $k=2$, then $P\in P_{k-2,m}$ can be written as $a\theta_{2,0} + b\theta_{2,1}$ for $a,b\in\CC$.
The dimension of $W_{\frac52-k, m,\chi_i}$ is zero when $i = 1,5,13$ and $17$, and it is $1$ when $i =9$ and $21$.
The following table shows the basis of $W_{\frac52-k, m,\chi_i}$ when $i=9$ and $21$.
\begin{table}[ht]\footnotesize\tabcolsep=0.21cm
\caption {Basis of $W_{\frac12, m,\chi_i}$ }
\begin{tabular}{c c c c c c c}
\  & $a$ & $b$  \\ \hline \noalign{\smallskip}
$\chi_9$ & $-\sqrt{2}+1$ & $1$\\
$\chi_{21}$ &  $\sqrt{2}+1$           &  $1$
\end{tabular}
\end{table}

If $k=3$, then $W_{\frac52-k, m,\chi_i}$ is zero for all possible multiplier systems.
If $k=4$, then $P\in P_{k-2,m}$ can be written as $(a_2\tau^2 + a_1\tau +a_0)\theta_{2,0}(\tau,z)+ (b_2\tau^2 + b_1\tau + b_0)\theta_{2,1}(\tau,z)$ for $a_i,b_i\in\CC$.
The dimension of $W_{\frac52-k, m,\chi_i}$ is $1$ for $i = 1,5,9,13,17$ and $21$.
The following table shows the basis of $W_{\frac52-k, m,\chi_i}$.
\begin{table}[ht]\footnotesize\tabcolsep=0.21cm
\caption {Basis of $W_{-\frac32, m,\chi_i}$}
\begin{tabular}{c c c c c c c}
\  & $a_2$ & $a_1$ & $a_0$ & $b_2$ & $b_1$ & $b_0$ \\ \hline \noalign{\smallskip}
$\chi_1$ & $-37.840070\cdots  $ & $(22.519313 \cdots)i$ & $52.513940\cdots$ & $-36.425856\cdots$ & $(9.3278049 \cdots)i$ & $1$\\  \noalign{\smallskip}

$\chi_5$ & $0.69364166\cdots$ & $(0.17762531\cdots )i$ & $-0.019042563\cdots$ & $-0.72057191\cdots$ & $-(0.42882543 \cdots)i$ & $1$\\  \noalign{\smallskip}

$\chi_9$ & $-1$ & $i$ & $0.41421356\cdots$ & $0.41421356\cdots$ & $(0.41421356\cdots)i$ & $1$ \\  \noalign{\smallskip}

$\chi_{13}$ & $0.51956178\cdots$ & $-(0.66290669  \cdots)i$ & $-0.26522868\cdots$           & $-0.89465178\cdots$ & $(1.6003983 \cdots)i$  & $1$ \\  \noalign{\smallskip}

$\chi_{17}$ & $-3.3731336\cdots$ & $-(6.0340318  \cdots)i$ & $3.7703313\cdots$ & $-1.9589200\cdots$ & $-(2.4993778\cdots)i$ & $1$\\  \noalign{\smallskip}

$\chi_{21}$ & $-1$ & $i$ & $-2.4142136\cdots$ & $-2.4142136\cdots$ & $-(2.4142136\cdots)i$ & $1$
\end{tabular}
\end{table}

By the definition of the pairing $\{\ ,\ \}$, we can compute the value in (\ref{peter}) using the coefficients appearing in period functions of $F'$ and $G'$.
For example, by Theorem 3.4 in \cite{EZ}, one can check that $\dim S_{4, 1,\chi_{triv}}(\Gamma^J)$ is zero, where $\chi_{triv}$ denotes the trivial character. This implies that $\dim S_{\frac52,1,\chi_{21}}(\Gamma^J)$  is also zero because the multiplication by $\eta^3$ gives an injective map from $S_{\frac52,1,\chi_{21}}(\Gamma^J)$ to $S_{4, 1,\chi_{triv}}(\Gamma^J)$.
 Therefore, there is a skew-holomorphic Jacobi cusp form $F'$ such that the corresponding period function $P(\widetilde{\hat{\mathcal{E}}^J}(F'))$   is $(\sqrt{2}+1)\theta_{2,0} + \theta_{2,1}$ by the isomorphism (\ref{iso2}) (for the definition of $P(\widetilde{\hat{\mathcal{E}}^J}(F'))$  see Section \ref{section4}). Then
\[
-24i(F',F') = \{F^+,F^+\} = -(4\sqrt{2}+4)i,
\]
where $F^+$ is a mock Jacobi form whose shadow is $F'$.
By the same way, there is a skew-holomorphic Jacobi cusp form $F'$ whose period function is
\begin{align*}
&(-(3.3731336\cdots)\tau^2 -(6.0340318\cdots)i\tau + 3.7703313\cdots)\theta_{2,0}(\tau,z)\\ &\qquad+(-(1.9589200\cdots)\tau^2 -(2.4003778\cdots)i\tau +1)\theta_{2,1}(\tau,z).
\end{align*}
Then
\[96i(F',F') = \{F^+,F^+\} = (100.656302\cdots)i.\]
We can also find a skew-holomorphic Jacobi cusp form $F'$ whose period function is
\[
(-\tau^2+i\tau-(2.4142136\cdots))\theta_{2,0}(\tau,z)+ (-(2.4142136\cdots)\tau^2-(2.4142136\cdots)i\tau+1)\theta_{2,1}(\tau,z).
\]
In this case we obtain
\[96i(F',F') = \{F^+,F^+\} = (14.485281\cdots)i.\]
\end{ex}

This paper is organized as follows. In Section \ref{section2}, we introduce
vector-valued modular forms and vector-valued harmonic weak Maass forms.
In Section \ref{section3},
we give basic notions of Jacobi forms, skew-holomorphic Jacobi forms and  harmonic Maass-Jacobi forms.
In Section \ref{section4}, we prove the main theorems:
Theorem \ref{main2} and \ref{Haberland}.

\section{Vector-valued modular forms} \label{section2}
In this section, we recall basic notions of vector-valued modular forms and vector-valued harmonic weak Maass forms. For details, consult \cite{KM0} for vector-valued modular forms and \cite{BF, BOR} for vector-valued harmonic weak Maass forms.

\subsection{Vector-valued modular forms} \label{section2.1}
We start with the definition of vector-valued modular forms.
First, we set up some notations.
Let $\Gamma\subset\Gamma_1$ be a $H$-group. Let $k\in\ZZ$ and let $\chi$ be a character of $\Gamma$. We suppose that  $[\Gamma_1:\Gamma]<\infty$.
Let $p$ be a positive integer and $\rho:\Gamma\to \GL(p,\CC)$ a $p$-dimensional unitary complex representation.  We denote the standard basis elements of the vector space $\CC^p$ by $\mbf{e}_j$ for $1\leq j\leq p$.
If $f = \sum_{j=1}^p f_j \mbf{e}_j$ is a vector-valued function on $\HH$, then the slash operator $|_{k,\chi,\rho} \gamma$ is defined by
\[
(f|_{k,\chi,\rho}\gamma)(\tau) := \chi(\gamma)^{-1}(c\tau+d)^{-k}\rho^{-1}(\gamma)f(\gamma\tau)
\]
for $\gamma = \sm a&b\\c&d\esm \in \Gamma$, where $\gamma\tau := \frac{a\tau+b}{c\tau+d}$.

\begin{dfn} \label{dfnofvvm} A vector-valued weakly holomorphic modular form of weight $k$, character $\chi$ and type $\rho$ on $\Gamma$ is a sum $f = \sum_{j=1}^p f_j \mbf{e}_j$ of  holomorphic functions in $\HH$ satisfying the followings.
\begin{enumerate}
\item For all $\gamma \in\Gamma$,
$f|_{k,\chi,\rho} \gamma = f$.

\item For each $\gamma = \sm a&b\\c&d\esm \in \Gamma_1$ the function $(c\tau+d)^{-k}f(\gamma\tau)$ has the Fourier expansion of the form
\[
\sum_{j=1}^p\sum_{n\gg-\infty} a_{j,\gamma}(n)e^{2\pi i(n+\kappa_{j,\gamma})\tau/\lambda_{\gamma}}\mbf{e}_j,
\]
where $\kappa_{j,\gamma}$ (resp. $\lambda_{\gamma}$) is a constant which depends on $j$ and $\gamma$ (resp. $\gamma$).
\end{enumerate}
\end{dfn}

We denote by $M^!_{k,\chi,\rho}(\Gamma)$  the space of all vector-valued weakly holomorphic modular forms of weight $k$, character $\chi$ and type $\rho$ on $\Gamma$. There are subspaces $M_{k,\chi,\rho}(\Gamma)$ and $S_{k,\chi,\rho}(\Gamma)$ of vector-valued modular forms and vector-valued cusp forms, respectively, for which we require that each $a_{j,q}(n) = 0$ when $n+\kappa_{j,q}$ is negative, respectively, non-positive.

Now we introduce the definition of vector-valued harmonic weak Maass forms and related differential operators. To define vector-valued harmonic weak Maass forms, we need  the weight $k$ hyperbolic Laplacian given by
\[
\Delta_k := -v^2\biggl(\frac{\partial^2}{\partial u^2}+\frac{\partial^2}{\partial v^2}\biggr) + ikv\biggl(\frac{\partial}{\partial u}+i\frac{\partial}{\partial v}\biggr).
\]

\begin{dfn} \label{dfnofhar} A vector-valued harmonic weak Maass form of weight $k$, character $\chi$ and type $\rho$ on $\Gamma$ is a sum $f = \sum_{j=1}^p f_j\mbf{e}_j$ of smooth functions on $\HH$ satisfying
\begin{enumerate}
\item[(1)] $f|_{k,\chi,\rho}\gamma = f$ for all $\gamma\in\Gamma$,
\item[(2)] $\Delta_k f =0$,
\item[(3)] there is a constant $C>0$ such that
\[
(c\tau+d)^{-k}f_j(\gamma\tau) = O(e^{Cv}),
\]
as $v\to\infty$ uniformly in $u$ for every integer $1\leq j\leq p$ and every element $\gamma=\sm a&b\\c&d\esm\in \Gamma_1$.
\end{enumerate}
We write $H_{k,\chi,\rho}(\Gamma)$ for the space of vector-valued harmonic weak Maass forms of weight $k$, character $\chi$ and type $\rho$ on $\Gamma$.
\end{dfn}

It is  known that any harmonic weak Maass form $f$  has a unique decomposition $f = f^+ + f^-$ (see \cite[Section 3]{BF}). Then the first (resp. second) summand is called the holomorphic (resp. non-holomorphic) part of $f$.
In the theory of vector-valued harmonic weak Maass forms, there are two important differential operators $\xi_{-k}$ and $D^{k+1}$. The differential operator $\xi_{-k}:= 2iv^{-k}\overline{\partial_{\overline{\tau}}}$ gives an anti-linear map (see Proposition 3.2  in \cite{BF})
\[
\xi_{-k} : H_{-k,\chi,\rho}(\Gamma) \to M^!_{k+2,\bar{\chi},\bar{\rho}}(\Gamma).
\]
Another differential operator $D^{k+1} := ( \frac{1}{2\pi i}\partial_{\tau})^{k+1}$
 gives  a linear  map (see Theorem 1.2 in \cite{BOR})
\[
D^{k+1}: H^*_{-k,\chi,\rho}(\Gamma) \to M^!_{k+2,\chi,\rho}(\Gamma).
\]
We let $H^*_{-k,\chi,\rho}(\Gamma)$ denote the inverse image of the space of cusp forms  under the map $\xi_{-k}$.

\subsection{Eichler integrals and harmonic weak Maass forms} \label{section2.3}
In this subsection, we recall the basic notion of Eichler integrals.
We denote the space of holomorphic parts of $f\in H^*_{-k,\chi,\rho}(\Gamma)$ by $H^+_{-k,\chi,\rho}(\Gamma)$.
On the other hand, we let $E_{-k,\chi,\rho}(\Gamma)$ denote the space of holomorphic vector-valued Eichler integrals $f$ of weight $-k$, character $\chi$ and type $\rho$ on $\Gamma$ such that
\begin{enumerate}
\item[(1)] vector-valued functions $f$ are invariant under $|_{-k,\chi,\rho}Q$,
\item[(2)] vector-valued functions $D^{k+1}(f)$ can be written as
\[
D^{k+1}(f) = g^* + D^{k+1}(h),
\]
where $g^*$ is a supplementary function to a cusp form $g\in S_{k+2,\bar{\chi},\bar{\rho}}(\Gamma)$ and $h \in M^!_{-k,\chi,\rho}(\Gamma)$ (for the definition of the supplementary function see \cite[Section 2.3]{CL}).
\end{enumerate}

We have two different Eichler integrals associated with a vector-valued cusp form $f$.
For $f\in S_{k,\chi,\rho}(\Gamma)$ define
\[
\mathcal{E}(f)(\tau) := \int^{i\infty}_{\tau} f(t)(t-\tau)^{k-2}dt
\]
and
\[
\hat{\mc{E}}(f)(\tau) := \overline{\int^{i\infty}_{\tau}f(t)(t-\overline{\tau})^{k}dt}.
\]
A holomorphic Eichler integral $\mathcal{E}(f)$ can be extended to a vector-valued weakly holomorphic modular form $f$ by considering term-by-term integration of the Fourier expansion of $f$ (for more details see \cite[Section 2.3]{CL}).





\section{Jacobi forms and harmonic Maass-Jacobi forms } \label{section3}

In this section, we review basic notions of Jacobi forms (see \cite{EZ, Z0}),  skew-holomorphic Jacobi forms (see \cite{Sko1,Sko2}) and harmonic Maass-Jacobi forms (see \cite{BR}).

\subsection{Jacobi forms} \label{section3.1}
In this subsection, we introduce Jacobi forms and theta expansions. First we fix some notations.
Let $\Gamma\subset\Gamma_1$ be a $H$-group with $[\Gamma_1:\Gamma]<\infty$.
Then $\Gamma^{J}$ acts on $\HH\times\CC$ as a group of automorphisms
\[
(\gamma, (\lambda,\mu))\cdot(\tau,z)= \biggl(\gamma\tau,\frac{z+\lambda\tau+\mu}{c\tau+d}\biggr)
\]
for $(\gamma, (\lambda,\mu))\in \Gamma^J$ and $(\tau,z)\in\HH\times\CC$.
Let $k\in\frac12\ZZ$  and let $\chi$ be a multiplier system of weight $k$ on $\Gamma$.
For $\gamma=\sm a&b\\c&d\esm \in\Gamma , X = (\lambda, \mu)\in \ZZ^2$ and $m\in\ZZ$ with $m>0$, we define
\[
(F|_{k,m,\chi} \gamma)(\tau,z) := (c\tau+d)^{-k}\chi(\gamma)^{-1}e^{-2\pi im\frac{cz^2}{c\tau+d}}F(\gamma(\tau,z))
\]
and
\[
(F|_{m}X)(\tau,z) :=e^{2\pi im(\lambda^2\tau+ 2\lambda z+\mu\lambda)}F(\tau,z+\lambda\tau+\mu),
\]
where $\gamma(\tau,z) = (\gamma\tau,\frac{z}{c\tau+d})$.
Then $\Gamma^{J}$ acts on the space of  functions on $\HH\times\CC$ by
\begin{equation} \label{slashoperator}
F|_{k,m,\chi} (\gamma,X) := F|_{k,m,\chi} \gamma |_{m}X.
\end{equation}
With these notations, we introduce the definition of a Jacobi form.

\begin{dfn}
A weak holomorphic Jacobi form of weight $k$, index $m$ and multiplier system $\chi$ on $\Gamma^J$  is a holomorphic function $F$ on $\HH\times\CC$ satisfying
\begin{enumerate}
\item[(1)] $F|_{k,m,\chi} (\gamma,X) =F$ for every $(\gamma,X)\in\Gamma^J$,
\item[(2)] for each $\gamma=\sm a&b\\c&d\esm \in\Gamma_1$, the function $(c\tau+d)^{-k}e^{2\pi im\frac{-cz^2}{c\tau+d}}F((\gamma,0)\cdot(\tau,z))$ has the Fourier expansion of the form
\begin{equation} \label{Jacobifourier}
\sum_{l,r\in\ZZ\atop 4(l+\kappa_\gamma)-\lambda_\gamma mr^2 \gg -\infty}a(l,r)e^{2\pi i(l+\kappa_{\gamma})/\lambda_\gamma}e^{2\pi irz}
\end{equation}
with suitable $0\leq \kappa_\gamma<1$ and $\lambda_\gamma\in\ZZ$.
\end{enumerate}
\end{dfn}

We denote by $J^!_{k,m,\chi}(\Gamma^{J})$ the  space of all weak holomorphic Jacobi forms of weight $k$, index $m$ and multiplier system $\chi$ on $\Gamma^{J}$.
If a weak holomorphic Jacobi form satisfies the  condition $a(l,r)\neq 0$ only if $4(l+\kappa_\gamma)-\lambda_\gamma mr^2 \geq0$ (resp. $4(l+\kappa_\gamma)-\lambda_\gamma mr^2 >0$), then it is called a Jacobi form (resp. Jacobi cusp form).
We denote by $J_{k,m,\chi}(\Gamma^J)$  the  space of all Jacobi forms  of weight $k$, index $m$ and multiplier system $\chi$ on $\Gamma^{J}$.

Now we introduce a theta series.
For a positive integer $m$ and $a\in\mathcal{N}$ consider
\[
\theta_{m,a}(\tau,z) := \sum_{\lambda\in\ZZ}e^{\pi im((\lambda+a)^2\tau+2(\lambda+a)z)},
\]
which converges normally on $\HH\times\CC$.
Then this theta series gives the following theorem.

\begin{thm} \cite[Section 5]{EZ}, \cite[Section 3]{Z0} \label{decomposition} Let $F(\tau,z)$ be a function on $\HH\times\CC$ which is  holomorphic as a function of $z$ and satisfies
\begin{equation} \label{elliptic}
F|_{m} X = F\
\end{equation}
for every $X\in \ZZ^2$.
Then
\begin{equation} \label{thetaexpansion}
F(\tau,z) = \sum_{a\in\mathcal{N}}f_a(\tau)\theta_{m,a}(\tau,z)
\end{equation}
with uniquely determined  functions $f_a:\HH\to\CC$.
\end{thm}

The decomposition by theta series as in (\ref{thetaexpansion}) is called the theta expansion.
This expansion induces an isomorphism between Jacobi forms and vector-valued modular forms.
 Let $F$ be a Jacobi form in $J_{k,m,\chi}(\Gamma^{J})$. By Theorem \ref{decomposition}, we have a theta expansion
\[
F(\tau,z) = \sum_{a\in\mc{N}}f_a(\tau)\theta_{m,a}(\tau,z).
\]
Then a vector-valued function $\sum_{a\in\mc{N}} f_a \mbf{e}_a$ is a vector-valued modular form in  $M_{k-\frac 12, \chi', \rho'}(\Gamma)$ (for the definition of $\chi'$ and $\rho'$ see   \cite[Section 2]{CL0}).

\begin{thm} \cite[Section 5]{EZ} \label{isomorphism}
The theta expansion  gives an isomorphism between $J_{k,m,\chi}(\Gamma^{J})$ and $M_{k-\frac 12, \chi', \rho'}(\Gamma)$. Furthermore, this isomorphism sends Jacobi cusp forms to vector-valued cusp forms.
\end{thm}

\subsection{Skew-holomorphic Jacobi forms} \label{section3.2}
In this subsection, we introduce skew-holomorphic Jacobi forms (for more details see \cite{BR,Sko1,Sko2}).
First we start with the definition of skew-holomorphic Jacobi forms.
For  $k\in\frac12\ZZ$ and a positive integer $m$, we have the  slash operator on functions $F:\HH\times\CC\to\CC$ defined by
\begin{eqnarray*}
(F|_{k,m,\chi}^{\sk} (\gamma,X))(\tau,z) &:=& \chi(\gamma)^{-1}(c\bar{\tau}+d)^{1-k}|c\tau+d|^{-1}e^{2\pi im(-\frac{c(z+\lambda\tau+\mu)^2}{c\tau+d}+\lambda^2\tau+2\lambda z)} F((\gamma,X)\cdot(\tau,z))
\end{eqnarray*}
for $(\gamma,X) = (\sm a&b\\c&d\esm, (\lambda,\mu))\in\Gamma^{J}$, where $\chi$ is a multiplier system of weight $k$ on $\Gamma$.

\begin{dfn}
 A function $F:\HH\times\CC\to\CC$ is a weak skew-holomorphic Jacobi form of weight $k$, index $m$ and multiplier system $\chi$ on $\Gamma^{J}$ if $F$ is real-analytic in $\tau\in\HH$ and holomorphic in $z\in\CC$, and it satisfies the  conditions
\begin{enumerate}
\item[(1)] for all $(\gamma,X)\in\Gamma^{J}$, $F|_{k,m,\chi}^{\sk}(\gamma,X) = F$,
\item[(2)] for each $\gamma=\sm a&b\\c&d\esm \in\Gamma_1$, the function $(c\tau+d)^{-k}e^{2\pi im(-\frac{cz^2}{c\tau+d})}F((\gamma,0)\cdot(\tau,z))$ has the Fourier expansion of the form
\begin{equation} \label{Jacobifourier}
\sum_{l,r\in\ZZ\atop 4m(l+\kappa_\gamma)-\lambda_\gamma r^2\ll \infty}a(l,r)e^{\frac{-\pi(\lambda_\gamma r^2-4m(l+\kappa_{\gamma}))v}{m\lambda_\gamma}}e^{2\pi i\tau(l+\kappa_{\gamma})/\lambda_\gamma}e^{2\pi irz}
\end{equation}
with  suitable $0\leq \kappa_\gamma<1$ and  $\lambda_\gamma\in\ZZ$.
\end{enumerate}
\end{dfn}

If the Fourier expansion in (\ref{Jacobifourier}) is only over $4m(l+\kappa_\gamma)-\lambda_\gamma r^2\leq0$ (resp. $4m(l+\kappa_\gamma)-\lambda_\gamma r^2<0$), then $F$ is called a skew-holomorphic Jacobi form (resp. skew-holomorphic Jacobi cusp form). We denote the space of weak skew-holomorphic Jacobi forms (resp. skew-holomorphic Jacobi forms)  of weight $k$, index $m$ and multiplier system $\chi$ on $\Gamma^{J}$ by $J^{\sk!}_{k,m,\chi}(\Gamma^{J})$ (resp. $J^{\sk}_{k,m,\chi}(\Gamma^{J})$).

It is known that skew-holomorphic Jacobi forms are closely related with vector-valued modular forms.
As in the case of a Jacobi form, a skew-holomorphic Jacobi form $F\in J^{\sk}_{k,m,\chi}(\Gamma^{J})$ has a theta expansion
\[
F(\tau,z) = \sum_{a\in \mc{N}}f_a(\tau)\theta_{m,a}(\tau,z).
\]
But in this case, $\sum_{a\in\mc{N}}\overline{f_a}\mbf{e}_a$ is a vector-valued modular form in $M_{k-\frac12, \overline{\chi'},\overline{\rho'}}(\Gamma)$, where $\chi'$ and $\rho'$ are as in Theorem \ref{isomorphism}.
This gives an isomorphism between skew-holomorphic Jacobi forms and vector-valued modular forms.

\begin{thm} \cite[Section 6]{BR}  \label{isomorphismskew}
The theta expansion  gives an isomorphism between $J^{\sk}_{k,m,\chi}(\Gamma^{J})$ and $M_{k-\frac 12, \overline{\chi'}, \overline{\rho'}}(\Gamma)$. Furthermore, this isomorphism sends skew-holomorphic Jacobi cusp forms to vector-valued cusp forms.
\end{thm}

With this isomorphism, we define a supplementary function in the case of Jacobi forms as follows. For $F\in S_{k,m,\chi}(\Gamma^{J})$ we have a theta expansion
\begin{equation} \label{skewcase}
F(\tau,z) = \sum_{a\in\mc{N}} f_a(\tau)\theta_{m, a}(\tau,z).
\end{equation}
Then a vector-valued function $f:=\sum_{a\in\mc{N}}f_a\mbf{e}_a$ is a vector-valued cusp form in $S_{k-\frac12, \chi',\rho'}(\Gamma)$.
By the vector-valued supplementary function theory  there is a vector-valued weakly holomorphic modular form $f^* = \sum_{a\in\mc{N}} f^*_a\mbf{e}_a \in M^!_{k-\frac12, \overline{\chi'},\overline{\rho'}}(\Gamma)$, which is a function supplementary to $f$ (for more details see \cite[Section 2.3]{CL}). Then we have a weak skew-holomorphic Jacobi form $F^*\in J^{\sk!}_{k,m,\chi}(\Gamma^{J})$ defined by
\[
F^*(\tau,z) = \sum_{a\in\mc{N}} \overline{f^*_a(\tau)}\theta_{m,a}(\tau,z).
\]
We call $F^*$ the function supplementary to  $F$.

Now we introduce how to construct a Jacobi integral if a skew-holomorphic Jacobi cusp form is given. Let $F\in J^{\sk,\cusp}_{k+\frac52, m,\chi}(\Gamma^J)$. We have a theta expansion as in (\ref{skewcase}).
Then we have a non-holomorphic Jacobi integral associated with $F$, which is given by
\[
\hat{\mc{E}}^J(F)(\tau,z) := \sum_{a\in\mc{N}} \hat{\mc{E}}(\overline{f_a})(\tau)\theta_{m,a}(\tau,z).
\]

\subsection{Harmonic Maass-Jacobi forms} \label{section3.3}
In this subsection, we introduce the theory of harmonic Maass-Jacobi forms. For details, consult \cite{BR}.

\begin{dfn} Let $k\in\frac12\ZZ$ and $\chi$ be a multiplier system of weight $k$ on $\Gamma$.
A function $F:\HH\times\CC\to\CC$ is a harmonic Maass-Jacobi form of weight $k$, index $m$ and multiplier system $\chi$ on $\Gamma^{J}$ if $F(\tau,z)$ is real-analytic in $\tau\in\HH$ and holomorphic in $z\in\CC$, and it satisfies the  conditions
\begin{enumerate}
\item[(1)] for all $(\gamma,X)\in \Gamma^{J}$, $F|_{k,m,\chi}(\gamma,X) = F$,
\item[(2)] $C^{k,m}(F) = 0$,
\item[(3)] for each $\gamma=\sm a&b\\c&d\esm \in\Gamma_1$
\[
(c\tau+d)^{-k}e^{2\pi im(-\frac{cz^2}{c\tau+d})}F((\gamma,0)\cdot(\tau,z)) = O(e^{av}e^{2\pi my^2/v})
\]
as $v\to\infty$ for some $a>0$.
\end{enumerate}
\end{dfn}

Note that elements $F$ in $\hat{J}^{\cusp}_{k,m,\chi}(\Gamma^J)$ have the Fourier expansion of the form
\begin{eqnarray} \label{fourier}
F(\tau,z) &=& \sum_{n,r\in\ZZ\atop 4(l+\kappa)m-r^2\lambda\gg-\infty}a^+(l,r)e^{2\pi i(l+\kappa)\tau/\lambda}e^{2\pi irz}\\
\nonumber&& + \sum_{l,r\in\ZZ\atop  4(l+\kappa)m-r^2\lambda<0} a^-(l,r)\Gamma\biggl(\frac 32-k,\frac{\pi(r^2\lambda-4(l+\kappa)m)v}{m\lambda}\biggr)e^{2\pi i(l+\kappa)\tau/\lambda}e^{2\pi irz},
\end{eqnarray}
where $\Gamma(\alpha, t) := \int^\infty_t e^{-w}w^{\alpha-1}dw$ is the incomplete Gamma function.
The first (resp. second) summand is called the holomorphic (resp. non-holomorphic) part of $F$ (for more details see \cite[Section 4]{BR}).

\begin{rmk}
For the Fourier expansion of a harmonic Maass-Jacobi form $F\in \hat{J}^{\cusp}_{k,m,\chi}(\Gamma^J)$ in (\ref{fourier}) we follow the notation in \cite{BR} in the case where $\xi_{k}(F)$ is a skew-holomorphic Jacobi cusp form and $m$ is a positive integer.
A more detailed treatment of harmonic Maass-Jacobi forms was given by Bringmann, Raum and Richter \cite{BRR}.
\end{rmk}

If $F\in \hat{J}^{\cusp}_{k,m,\chi}(\Gamma^J)$, then $F$ has a theta expansion
\begin{equation} \label{thetaexpansionharmonic}
F(\tau,z) = \sum_{a\in\mc{N}}f_a(\tau)\theta_{m,a}(\tau,z).
\end{equation}
By the insightful observation by Bringmann and Richter \cite{BR}, we see that $\sum_{a\in\mc{N}} f_a\mbf{e}_a$ is a vector-valued harmonic weak Maass form in $H^*_{k-\frac12, \chi',\rho'}(\Gamma)$.
The theta expansion has further interesting properties.
The following lemma shows that two operators $\xi_{\frac52-k,m}$ and $H_m^{k-1}$ commute with the theta expansion.

\begin{lem} \label{commute}
Let $F\in \hat{J}^{\cusp}_{\frac52-k,m,\chi}(\Gamma^J)$ with a theta expansion as in (\ref{thetaexpansionharmonic}).
\begin{enumerate}
\item The function $\xi_{\frac52-k, m,\chi}(F)\in J^{\sk,\cusp}_{k+\frac12,m,\chi}(\Gamma^J)$ has a theta expansion
\[
\xi_{\frac52-k, m,\chi}(F)(\tau,z) = \sum_{a\in\mc{N}} \overline{\xi_{2-k}(f_a)(\tau)}\theta_{m,a}(\tau,z).
\]

\item The function $H_m^{k-1}(F)\in J^!_{k+\frac12, m,\chi}(\Gamma^J)$ has a theta expansion
\[
H_m^{k-1}(F)(\tau,z) = (-16\pi^2m)^{k-1}\sum_{a\in\mc{N}}D^{k-1}(f_a)(\tau)\theta_{m,a}(\tau,z).
\]
\end{enumerate}
\end{lem}

\begin{proof} [\bf Proof of Lemma \ref{commute}]
Using the argument in \cite{BR} for the case of $\Gamma=\Gamma_1$, we can prove (1) for general $\Gamma$.
If we use the fact that $H_{m}(\theta_{m,a}) = 0$, then
\[
H_m\biggl(\sum_{a\in\mc{N}}f_a(\tau)\theta_{m,a}(\tau,z)\biggr) = (-16\pi^2 m)\sum_{a\in\mc{N}}D(f_a)(\tau)\theta_{m,a}(\tau,z).
\]
From this we can prove (2).
\end{proof}

From this, one can see that the heat operator $H^{k-1}_m$ gives a linear map between harmonic Maass-Jacobi forms and weak holomorphic Jacobi forms and this map depends only on the holomorphic parts of harmonic Maass-Jacobi forms.

\begin{lem} \label{main1}
The heat operator $H^{k-1}_m$ gives a linear map
\[
H_m^{k-1} : \hat{J}^{\cusp}_{\frac52-k,m,\chi}(\Gamma^J) \to J^{!}_{k+\frac12,m,\chi}(\Gamma^J).
\]
Moreover, if $F \in \hat{J}^{\cusp}_{\frac52-k,m,\chi}(\Gamma^J)$, then
\[
H_m^{k-1}(F) = H_m^{k-1}(F^+).
\]
\end{lem}

This is an analogous result of Theorem 1.1 in \cite{BOR}. We prove Lemma \ref{main1} by using the Bol's identity in \cite{Bol} and \cite{CK} and the theta expansion.

\begin{rmk}
Authors of  \cite{BRR}   proved that the heat operator gives  analogue of $\xi_k$ on the space of skew-holomorphic Jacobi forms and calculated the action of the heat operator on Fourier expansions of skew-holomorphic Jacobi forms.
\end{rmk}

\begin{proof} [\bf Proof of Lemma \ref{main1}]
Suppose that $F \in \hat{J}^{\cusp}_{\frac52-k,m,\chi}(\Gamma^J)$. We have a theta expansion
\[
F(\tau,z) = \sum_{a\in\mc{N}}f_a(\tau)\theta_{m,a}(\tau,z).
\]
Then by Lemma \ref{commute} we have
\begin{equation} \label{formulaforheat}
H^{k-1}_m(F)(\tau,z) = (-16\pi^2m)^{k-1}\sum_{a\in\mc{N}}D^{k-1}(f_a)(\tau)\theta_{m,a}(\tau,z).
\end{equation}
Since the vector-valued function  $\sum_{a\in\mc{N}} f_a\mbf{e}_a$ is in $H^*_{2-k, \chi',\rho'}(\Gamma)$,  the vector-valued function
\[
\sum_{a\in\mc{N}} D^{k-1}(f_a)\mbf{e}_a
\]
is a weakly holomorphic modular form in $M^!_{k,\chi',\rho'}(\Gamma)$.
Therefore, $H^{k-1}_m(F)$ is a weak holomorphic Jacobi form in $J^{!}_{k+\frac12,m,\chi}(\Gamma^J)$. Moreover, one can see that
\[
F^+(\tau,z) = \sum_{a\in\mc{N}}f_a^+(\tau)\theta_{m,a}(\tau,z).
\]
By (\ref{formulaforheat}) and Theorem 1.1 in \cite{BOR} we see that $H^{k-1}_m(F)$ depends only on
the holomorphic part $F^+$.
\end{proof}




\section{Proof of the main theorems} \label{section4}
\begin{proof} [\bf Proof of Theorem \ref{main2}]
The proof is based on the theta expansion in Theorem \ref{decomposition} and the argument in \cite{CKL}.
It is clear that $C_{m,\chi}$ is contained in $E_{\frac52-k, m,\chi}(\Gamma^J)$.

If $F\in E_{\frac52-k, m,\chi}(\Gamma^J)$, then it has a theta expansion
\begin{equation} \label{EichlerJacobiintegral}
F(\tau,z) = \sum_{a\in\mc{N}} f_a(\tau)\theta_{m,a}(\tau,z).
\end{equation}
Then one can check that a vector-valued function $f := \sum_{a\in\mc{N}} f_a\mbf{e}_a$ is a vector-valued Eichler integral in $E_{2-k,\chi',\rho'}(\Gamma)$. This comes from Lemma 3.6 in \cite{CL1} and the property of the theta expansion:
for $(\gamma,X)\in\Gamma^{J}$,
\begin{equation} \label{propertyofthetaexpansion}
\biggl(\biggl(\sum_{a\in\mc{N}}f_a\theta_{m,a}\biggr)\biggr|_{\frac 52-k, m,\chi}(\gamma,X)\biggr)(\tau,z) = \sum_{a\in\mc{N}}\widetilde{f_a}(\tau)\theta_{m,a}(\tau,z),
\end{equation}
where $\sum_{a\in\mc{N}}\widetilde{f_a}\mbf{e}_a = \biggl(\sum_{a\in\mc{N}}f_a\mbf{e}_a\biggr)\biggr|_{2-k,\chi',\rho'}\gamma$.
Therefore, we have
\begin{equation} \label{eq1}
D^{k-1}(f) = g^* + D^{k-1}(h),
\end{equation}
where $g^*$ is a supplementary function to a cusp form $g\in S_{k, \overline{\chi'},\overline{\rho'}}(\Gamma)$ and $h\in M^!_{2-k,\chi',\rho'}(\Gamma)$.
Define $b^+ = \mathcal{E}(g^*)$ and $b^- = -\hat{\mathcal{E}}(g)$.
Then by Theorem 2.8 and 2.9 in \cite{CL} one can see that $b:= b^+ + b^-$ is invariant under the slash operator $|_{2-k,\chi',\rho'}\gamma$ for alll $\gamma\in\Gamma$.
We can also check that $\Delta_{2-k}(b) = 0$ and hence $h\in H^*_{2-k,\chi',\rho'}(\Gamma)$.
Note that
\[
D^{k-1}(b) = D^{k-1}(b^+) = \frac{(-1)^{k-1}(k-2)!}{(2\pi i)^{k-1}}g^*.
\]
From this the function $\frac{(-1)^{k-1}(2\pi i)^{k-1}}{(k-2)!}b+h$ is an element of $H^*_{2-k,\chi',\rho'}(\Gamma)$ and its image under $D^{k-1}$ is the same as $D^{k-1}(f)$ by (\ref{eq1}).
Therefore, we have
\[
f= \left(\frac{(-1)^{k-1}(2\pi i)^{k-1}}{(k-2)!}b+h\right)^+ + c
\]
for some vector-valued constant $c$.

As we noted in Section \ref{section3.3}, the theta expansion induces an isomorphism between $\hat{J}^{\cusp}_{\frac52-k, m,\chi}(\Gamma^J)$ and $H^*_{2-k,\chi',\rho'}(\Gamma)$. This isomorphism sends holomorphic parts (resp. non-holomorphic parts) of vector-valued harmonic Maass-Jacobi forms to holomorphic parts (resp. non-holomorphic parts) of harmonic weak Maass forms.
Therefore, we have an isomorphism between $\hat{J}^+_{\frac52-k,m,\chi}(\Gamma^J)$ and $H^+_{2-k,\chi',\rho'}(\Gamma)$.
On the other hand, for a generator $Q$ of $\Gamma_\infty$ we have
\[
\chi'(Q)\rho'(Q) = \chi(Q) A_{Q,\chi} = B_{Q,\chi},
\]
where $A_{Q,\chi}$ and $B_{Q,\chi}$ are diagonal $|\mc{N}|$ by $|\mc{N}|$ matrices given by
$(A_{Q,\chi})_{a,a} = e^{-2\pi i\lambda ma^2}$ and $(B_{Q,\chi})_{a,a} = e^{2\pi i\kappa_{a}}$
for $a\in\mc{N}$. Here $\kappa_a$  is a real number with $0\leq \kappa_{a}<1$. Then one can see that the theta expansion gives an isomorphism between $C_{m,\chi}$ and $\sum_{a\in\mc{N}}\delta_{\kappa_{a},0}\CC$.
The first part of Theorem \ref{main2} (1) follows from two isomorphisms: $\hat{J}^+_{\frac52-k,m,\chi}(\Gamma^J) \cong H^+_{2-k,\chi',\rho'}(\Gamma)$ and $C_{m,\chi} \cong \sum_{a\in\mc{N}}\delta_{\kappa_{a},0}\CC$.

If $F\in \hat{J}^+_{\frac52-k,m,\chi}(\Gamma^J)$, then the corresponding vector-valued function $f^+ = \sum_{a\in\mathcal{N}} f^+_a \mbf{e}_a$
given by the theta expansion is in $H^+_{2-k,\chi',\rho'}(\Gamma)$.
For $f^+\in H^+_{2-k,\chi',\rho'}(\Gamma)$ there is a corresponding $f\in H^*_{2-k,\chi',\rho'}(\Gamma)$.
Let $g = \xi_{2-k}(f)$ and let $g^*$ be its supplementary function.
We already checked that $b:= \mathcal{E}(g^*) -\hat{\mathcal{E}}(g)$ is a harmonic weak Maass form in $H^*_{2-k,\chi',\rho'}(\Gamma)$.
Then by a direct computation we have
\[\xi_{2-k}\left(f+ \frac{(k-2)!}{(4\pi)^{k-1}}b\right) =0\]
and hence
\[h:= f+ \frac{(k-2)!}{(4\pi)^{k-1}}b\in M^!_{2-k,\chi',\rho'}(\Gamma).\]
From this we obtain
\begin{eqnarray*}
H_m^{k-1}(F) (\tau,z)&=&  (-16\pi^2m)^{k-1}\sum_{a\in\mathcal{N}}D^{k-1}(f_a)(\tau)\theta_{m,a}(\tau,z) \\
&=&  (-16\pi^2m)^{k-1} \sum_{a\in\mathcal{N}} D^{k-1}\left(-\frac{(k-2)!}{(4\pi)^{k-1}}b_a+h_a\right)(\tau)\theta_{m,a}(\tau,z)\\
&=&  (-16\pi^2m)^{k-1}\sum_{a\in\mathcal{N}} \left(\frac{(-1)^{k-2}((k-2)!)^2}{(8\pi^2 i)^{k-1}} g^*_a + D^{k-1}(h_a)\right)(\tau)\theta_{m,a}(\tau,z)\\
&=&    G^*     (\tau,z)                  +  H_m^{k-1}(H)(\tau,z),
\end{eqnarray*}
where $G^*(\tau,z) = -(-2mi)^{k-1}((k-2)!)^2\sum_{a\in\mathcal{N}} g^*_a(\tau)\theta_{m,a}(\tau,z)$ and $H(\tau,z) =  \sum_{a\in\mathcal{N}}  h_a(\tau)\theta_{m,a}(\tau,z)$. Here $G^*$ is a supplementary function to a cusp form
\[
G :=  -(-2mi)^{k-1}((k-2)!)^2\sum_{a\in\mathcal{N}} g_a(\tau)\theta_{m,a}(\tau,z).
\]
This completes the proof of Theorem \ref{main2} (1).



To prove Theorem \ref{main2} (2), let $F \in \hat{J}^{\cusp}_{\frac52-k, m,\chi}(\Gamma^J)$. Then by the part (1) we have
\begin{equation} \label{holomorphicpartexpression}
H^{k-1}_m (F^+) = G^*(\tau,z) + H^{k-1}_m (H) + \sum_{a\in\mc{N}_\chi} c(a)\theta_{m,a},
\end{equation}
where $G^*$ is a supplementary function to a skew-holomorphic Jacobi cusp form $G\in J^{\sk,\cusp}_{k+\frac12, m,\chi}(\Gamma^J)$, $H\in J^!_{\frac52-k, m,\chi}(\Gamma^J)$ and $c(a)\in \CC$. Each function in (\ref{holomorphicpartexpression}) has a theta expansion
\[
F(\tau,z) = \sum_{a\in\mc{N}} f_a(\tau)\theta_{m,a}(\tau,z),
\]
\[
G^*(\tau,z) = \sum_{a\in\mc{N}} g^*_a(\tau)\theta_{m,a}(\tau,z),
\]
and
\[
G(\tau,z) = \sum_{a\in\mc{N}} g_a(\tau)\theta_{m,a}(\tau,z).
\]
Then we see that
\[
\sum_{a\in\mc{N}}f_a\mbf{e}_a\in H^*_{2-k,\chi',\rho'}(\Gamma),
\]
\[
\sum_{a\in\mc{N}}g^*_a\mbf{e}_a\in M^!_{k,\chi', \rho'}(\Gamma),
\]
and
\[
\sum_{a\in\mc{N}}\overline{g_a}\mbf{e}_a\in S_{k,\overline{\chi'},\overline{\rho'}}(\Gamma).
\]
By the same argument in the proof of Theorem 1.1 in \cite{CKL} we obtain
\[\sum_{a\in\mc{N}}f^-_a\mbf{e}_a = \frac{(4\pi)^{k-1}}{c_k\Gamma(k-1)}\sum_{a\in\mc{N}} \hat{\mc{E}}(\overline{g_a})\mbf{e}_a,\]
where  $c_k := -\frac{\Gamma(k-1)}{(2\pi i)^{k-1}}$.
Therefore, we have
\begin{eqnarray*}
F^-(\tau,z) &=& \sum_{a\in\mc{N}} f^-_a(\tau)\theta_{m,a}(\tau,z)= \frac{(4\pi)^{k-1}}{c_k\Gamma(k-1)} \sum_{a\in\mc{N}} \hat{\mc{E}}(\overline{g_a})(\tau) \theta_{m,a}(\tau,z)\\
&=& \hat{\mc{E}}^J\biggl(\frac{(4\pi)^{k-1}}{\overline{c_k}\Gamma(k-1)} \sum_{a\in\mc{N}}g_a\theta_{m,a}\biggr)(\tau,z).
\end{eqnarray*}
This is the desired result.
\end{proof}

To prove Theorem \ref{Haberland} we need the lemma for period functions of a skew-holomorphic Jacobi cusp form.
For a skew-holomorphic Jacobi cusp form $F\in J^{\sk,\cusp}_{k+\frac12,m,\chi}(\Gamma^J)$, we have a theta expansion
\begin{eqnarray*}
F(\tau,z) &:=& \sum_{a\in\mathcal{N}} f_a(\tau)\theta_{m,a}(\tau,z).
\end{eqnarray*}
We define a  holomorphic vector-valued Jacobi integral associated with $F$ as a function $\widetilde{\mathcal{E}^J}(F): \Gamma\setminus\Gamma_1\to A_{\frac52-k, m,\chi}(\Gamma^J)$ given by
\[
\widetilde{\mathcal{E}^J}(F)(A)(\tau,z) := \sum_{a\in\mathcal{N}} \mathcal{E}(\overline{(f_A)_a})(\tau)\theta_{m,a}(\tau,z),
\]
where $\sum_{a\in\mathcal{N}} \overline{(f_A)_a}\mbf{e}_a = \left(\sum_{a\in\mathcal{N}} \overline{f_a}\mbf{e}_a\right)\bigg|_{k+2,\overline{\chi'},\overline{\rho'}}A$.
Then a period function of $\widetilde{\mathcal{E}^J}(F)$ is defined by
\[
P(\widetilde{\mathcal{E}^J}(F)) := \widetilde{\mathcal{E}^J}(F) - \widetilde{\mathcal{E}^J}(F)\|_{\frac52-k,m,\chi}(S,0).
\]
Similarly, we can define a non-holomorphic Jacobi integral and its period function associated with $F$ by
\[
\widetilde{\hat{\mathcal{E}}^J}(F)(A)(\tau,z) := \sum_{a\in\mathcal{N}} \hat{\mathcal{E}}(\overline{(f_A)_a})(\tau)\theta_{m,a}(\tau,z)
\]
and
\[
P(\widetilde{\hat{\mathcal{E}}^J}(F)) := \widetilde{\hat{\mathcal{E}}^J}(F) - \widetilde{\hat{\mathcal{E}}^J}(F)\|_{\frac52-k,m,\chi}(S,0).
\]

\begin{lem} \label{periodrelation} For every $A\in\Gamma\setminus\Gamma_1$ we have
\[
P(\widetilde{\mathcal{E}^J}(F))(A)^c = P(\widetilde{\hat{\mathcal{E}}^J}(F))(A).
\]
\end{lem}

\begin{proof} [\bf Proof of Lemma \ref{periodrelation}]
By a direct computation we obtain
\[
P(\widetilde{\mathcal{E}^J}(F))(A) = \sum_{a\in\mathcal{N}} \biggl( \int^{i\infty}_0 \overline{(f_A)_a}(t)(t-\tau)^{k-2}dt\biggr) \theta_{m,a}(\tau,z)
\]
and
\[
P(\widetilde{\hat{\mathcal{E}}^J}(F))(A)= \sum_{a\in\mathcal{N}} \biggl( \overline{\int^{i\infty}_0 \overline{(f_A)_a}(t)(t-\overline{\tau})^{k-2}dt}\biggr) \theta_{m,a}(\tau,z).
\]
This completes the proof.
\end{proof}

\begin{proof} [\bf Proof of Theorem \ref{Haberland}]
The proof is based on the arguments in \cite{EZ} and \cite{PP}.
Since $F'$ and $G'$ are skew-holomorphic Jacobi cusp forms, they have theta expansions
\begin{eqnarray*}
F'(\tau,z) &=& \sum_{a\in\mathcal{N}} f'_a(\tau)\theta_{m,a}(\tau,z),\\
G'(\tau,z) &=& \sum_{a\in\mathcal{N}} g'_a(\tau)\theta_{m,a}(\tau,z).
\end{eqnarray*}
By the proof of Theorem 5.3 in \cite{EZ}  we have
\[
\int_{\CC/\ZZ\tau+\ZZ} \theta_{m,a}(\tau,z)\overline{\theta_{m,b}(\tau,z)}e^{-4\pi my^2/v} dxdy = \sqrt{v/4m}\delta_{ab},
\]
where $\delta_{ab}$ is the Kronecker delta of $a$ and $b$.
Therefore, we obtain
\begin{eqnarray} \label{computationpetersson}
\nonumber (F,G) &=& \frac{1}{[\Gamma_1:\Gamma]} \int_{\Gamma\setminus\HH} \sum_{a,b\in\mathcal{N}}f'_a(\tau)\overline{g'_b(\tau)}
\int^{\CC/\ZZ\tau+\ZZ} \theta_{m,a}(\tau,z)\overline{\theta_{m,b}(\tau,z)}e^{-4\pi my^2/v} dxdy v^{k-\frac52}dudv\\
&=& \frac{1}{\sqrt{4m}}\frac{1}{[\Gamma_1:\Gamma]} \int_{\Gamma\setminus\HH} \sum_{a\in\mathcal{N}} f'_a(\tau)\overline{g'_a(\tau)} v^{k-2}dudv.
\end{eqnarray}
Note that $f'= \sum_{a\in\mathcal{N}}\overline{f'_a}\mbf{e}_a$ and $g' = \sum_{a\in\mathcal{N}}\overline{g'_a}\mbf{e}_a$ are vector-valued modular forms in $S_{k,\overline{\chi'},\overline{\rho'}}(\Gamma)$ by Theorem \ref{isomorphismskew} and
\[
(g',f')_M := \frac{1}{[\Gamma_1:\Gamma]} \int_{\Gamma\setminus\HH} \sum_{a\in\mathcal{N}}\overline{g'_a(\tau)} f'_a(\tau) v^{k-2}dudv
\]
is the Petersson inner product for vector-valued modular forms $f' $ and $g'$.

Let $\mathcal{D}$ be a fundamental domain for $\Gamma(2)$.
Then the region $\mathcal{D}$ consists of six copies of the fundamental domain for $\Gamma_1$.
Note that  the  slash operator $|_{k,\overline{\chi'},\overline{\rho'}}$ can be extended to $\Gamma_1$ because $\chi$ is a multiplier system on $\Gamma_1$ and $\rho'$ comes from the transformation properties of theta series $\theta_{m,a}$.
From this one can see that
\begin{eqnarray*}
6C_k[\Gamma_1:\Gamma](g',f')_M &=& \sum_{A\in \Gamma\setminus \Gamma_1} \int_{\mathcal{D}} [g'|_{k,\overline{\chi'},\overline{\rho'}}A(\tau), \overline{f'|_{k,\overline{\chi'},\overline{\rho'}}A(\tau)}] (\tau-\bar{\tau})^{k-2}d\tau d\bar{\tau},
\end{eqnarray*}
where $[g',f'] := \sum_{a\in\mathcal{N}}\overline{g'_a f'_a}$ and $C_k := -(2i)^{k-1}$.
We define
\[
G_A(\tau) := \int^{\tau}_{i\infty} (g'|_{k,\overline{\chi'},\overline{\rho'}}A)(t)(t-\bar{\tau})^{k-2}dt.
\]
Since
\[
\frac{\partial G_A}{\partial \tau}(\tau) = (g'|_{k,\overline{\chi'},\overline{\rho'}}A)(\tau) (\tau-\bar{\tau})^{k-2},
\]
by Stokes' theorem we have
\[
6C_k[\Gamma_1:\Gamma](g',f')_M =  \sum_{A\in \Gamma\setminus \Gamma_1} \int_{\partial \mathcal{D}} [G_A(\tau), \overline{f'|_{k,\overline{\chi'},\overline{\rho'}}A(\tau)}]d\bar{\tau}.
\]
The region $\mathcal{D}$ has vertices $i\infty, -1, 0$ and $1$.
Since
\[
\int^{i\infty}_1 [G_A(\tau), \overline{f'|_{k,\overline{\chi'},\overline{\rho'}}A(\tau)}]d\bar{\tau} = \int^{i\infty}_{-1} [G_{AT^2}(\tau), \overline{f'|_{k,\overline{\chi'},\overline{\rho'}}AT^2(\tau)}]d\bar{\tau},
\]
the sum of integrals over the vertical sides of $\mathcal{D}$ vanishes.
For other parts of $\partial \mathcal{D}$  we compute
\begin{eqnarray*}
\int^0_{-1} [G_A(\tau), \overline{f'|_{k,\overline{\chi'},\overline{\rho'}}A(\tau)}] d\bar{\tau} &=& \int^{i\infty}_1 [G_{AS}(\tau), \overline{f'|_{k,\overline{\chi'},\overline{\rho'}}AS(\tau)}]d\bar{\tau} \\
&&+ \int^{i\infty}_{1}\int^{i\infty}_{0} [g'|_{k,\overline{\chi'},\overline{\rho'}}AS(t),\overline{f'|_{k,\overline{\chi'},\overline{\rho'}}AS(\tau)}](t-\bar{\tau})^{k-2}dtd\tau,\\
\int^1_{0} [G_A(\tau), \overline{f'|_{k,\overline{\chi'},\overline{\rho'}}A(\tau)}] d\bar{\tau} &=& \int_{i\infty}^{-1} [G_{AS}(\tau), \overline{f'|_{k,\overline{\chi'},\overline{\rho'}}AS(\tau)}]d\bar{\tau}\\
&& - \int^{i\infty}_{-1}\int^{i\infty}_{0} [g'|_{k,\overline{\chi'},\overline{\rho'}}AS(t),\overline{f'|_{k,\overline{\chi'},\overline{\rho'}}AS(\tau)}](t-\bar{\tau})^{k-2}dtd\tau.
\end{eqnarray*}
When we add two equations and sum over $A\in\Gamma\setminus\Gamma_1$, the single integrals cancel as before and
\begin{align} \label{eq2}
\nonumber &6C_k[\Gamma_1:\Gamma](g',f')_M \\
& \qquad =\sum_{A\in\Gamma\setminus\Gamma_1} \left(\int^{i\infty}_{1}\int^{i\infty}_0 - \int^{i\infty}_{-1}\int^{i\infty}_0\right) [g'|_{k,\overline{\chi'},\overline{\rho'}}A(t), \overline{f'|_{k,\overline{\chi'},\overline{\rho'}}A(\tau)}] (t-\bar{\tau})^{k-2}dtd\bar{\tau}.
\end{align}
After a change of variable, the first integral becomes
\[
\bigg\langle \int^{i\infty}_0 \overline{f'|_{k,\overline{\chi'},\overline{\rho'}}AT(\tau)}(\overline{T\tau}-X)^{k-2}d\overline{\tau}, \int^{i\infty}_0 g'|_{k,\overline{\chi'},\overline{\rho'}}A(t)(t-X)^{k-2}dt \bigg\rangle.
\]
Therefore, we have
\[
\sum_{A\in\Gamma\setminus\Gamma_1}
\int^{i\infty}_{1}\int^{i\infty}_0  [g'|_{k,\overline{\chi'},\overline{\rho'}}A(t), \overline{f'|_{k,\overline{\chi'},\overline{\rho'}}A(\tau)}]
= \lla P(\widetilde{\mathcal{E}^J}(F'))^c\|_{\frac52-k, m,\chi} (T^{-1},0), P(\widetilde{\mathcal{E}^J}(G'))\rra.
\]
By a similar argument one can see that
\[
\sum_{A\in\Gamma\setminus\Gamma_1}
\int^{i\infty}_{-1}\int^{i\infty}_0  [g'|_{k,\overline{\chi'},\overline{\rho'}}A(t), \overline{f'|_{k,\overline{\chi'},\overline{\rho'}}A(\tau)}]
= \lla P(\widetilde{\mathcal{E}^J}(F'))^c\|_{\frac52-k, m,\chi} (T,0), P(\widetilde{\mathcal{E}^J}(G')) \rra.
\]
By lemma \ref{periodrelation} we obtain
\begin{eqnarray*}
6C_k[\Gamma_1:\Gamma](g',f')_M &=& \lla P(\widetilde{\mathcal{E}^J}(F'))^c\|_{\frac52-k,m,\chi}((T^{-1},0)-(T,0)), P(\widetilde{\mathcal{E}^J}(G'))\rra\\
& =&\lla P(\widetilde{\hat{\mathcal{E}}^J}(F'))\|_{\frac52-k,m,\chi}((T^{-1},0)-(T,0)), P(\widetilde{\hat{\mathcal{E}}^J}(G'))^c\rra .
\end{eqnarray*}

Note that for each $A\in\Gamma\setminus\Gamma_1$
\[\widetilde{\hat{\mathcal{E}}^J}(F')(A) = \hat{\mathcal{E}}^J(F'|^{\sk}_{k+\frac12,m,\chi}A) = F_A^-.
\]
We can also define $\widetilde{F^-}$ and its period function $P(F^-)$ by the same way in (\ref{plus1}) and (\ref{plus2}).
Then one can see that $P(F^+) = -P(F^-)$ since $F = F^+ + F^-$ is invariant under the slash operator $|_{\frac52-k,m,\chi}$.
Therefore, we have
\[
6C_k[\Gamma_1:\Gamma](f',g')_M = \lla P(F^+)\|_{\frac52-k,m,\chi}((T^{-1},0)-(T,0)), P(G^+)^c\rra.
\]
If we combine this with (\ref{computationpetersson}), then
\[
6C_k[\Gamma_1:\Gamma](F',G') = \frac{1}{\sqrt{4m}} \{F^+, G^+\}.
\]

Furthermore, by a direct computation of (\ref{eq2}) one can see that
\begin{align*}
&6C_k[\Gamma_1:\Gamma](g',f')_M\\
&\qquad =\sum_{A\in\Gamma\setminus\Gamma_1}\sum_{a\in\mathcal{N}} \sum_{0\leq m+n\leq k-2} \frac{(k-2)!}{(k-2-m-n)!}\frac{i^{m+n+2}}{(2\pi)^{m+n+2}}\\
&\qquad \times\left((-1)^{k-2-m} L(\overline{(g'_A)_a},n+1)\overline{L(\overline{(f'_AT)_a},m+1)}-(-1)^m L(\overline{(g'_A)_a},n+1)\overline{L(\overline{(f'_AT^{-1})_a},m+1)}\right) \\
& \qquad = \sum_{A\in\Gamma\setminus\Gamma_1}\sum_{a\in\mathcal{N}} \sum_{0\leq m+n\leq k-2\atop m\neq\equiv n\ (\m\ 2)} \frac{(k-2)!}{(k-2-m-n)!}\frac{i^{m+n}}{(2\pi)^{m+n+2}}\\
&\qquad \times\biggl((-1)^{k-2-m} L^{\sk}(G'|^{\sk}_{k+\frac12,m,\chi}A,a,n+1)\overline{L^{\sk}(F'|^{\sk}_{k+\frac12,m,\chi}AT,a,m+1)}\\
&\qquad \qquad-(-1)^m L^{\sk}(G'|^{\sk}_{k+\frac12,m,\chi}A,a,n+1)\overline{L^{\sk}(F'|^{\sk}_{k+\frac12,m,\chi}AT^{-1},a,m+1)}\biggr).
\end{align*}
Here $L(\overline{(f'_A)_a},s)$ is a function defined by the analytic continuation of the series
\[
\sum_{n+\kappa>0} \frac{a(n)}{\left(\frac{n+\kappa}{\lambda}\right)^s}
\]
when $\overline{(f'_A)_a}$ has the Fourier expansion of the form
\[
\overline{(f'_A)_a}(\tau) = \sum_{n+\kappa>0} a(n)e^{2\pi i(n+\kappa)/\lambda}
\]
for suitable real number $\kappa\in[0,1)$ and positive integer $\lambda$.
\end{proof}

\end{document}